\documentclass{article}

\usepackage{authblk}
\usepackage{amsmath, amssymb, amsthm,,enumerate,graphicx,pictexwd,dcpic}
\usepackage[lined,boxed,commentsnumbered]{algorithm2e}
\usepackage{url}

\newtheorem{theorem}{Theorem}[section]

\newtheorem{proposition}[theorem]{Proposition}
\newtheorem{corollary}[theorem]{Corollary}
\newtheorem{conjecture}{Conjecture}
\newtheorem*{conjecture*}{Conjecture}

\theoremstyle{definition}
\newtheorem{definition}[theorem]{Definition}
\newtheorem{example}[theorem]{Example}
\newtheorem{remark}[theorem]{Remark}

\numberwithin{equation}{section}

\newcommand{\zre}{\mathfrak{R}\mspace{1mu}}
\newcommand{\zdu}{\mathfrak{D}\mspace{1mu}}

\newcommand{\Z}{\mathfrak{Z}}
\newcommand{\I}{\mathfrak{I}}

\title{Zeon and Idem-Clifford Formulations of Hypergraph Problems}
\author{Samuel Ewing\footnote{Email: samuel.ewing@outlook.com}\, and G. Stacey Staples\footnote{Corresponding author.  Email: sstaple@siue.edu}}
\affil{Department of Mathematics \& Statistics\\
 Southern Illinois University Edwardsville\\
 Edwardsville, IL 62026-1653, USA}

\date{}

\begin{document}

\maketitle

\date{}

\begin{abstract}
Zeon algebras have proven to be useful for enumerating structures in graphs, such as paths, trails, cycles, matchings, cliques, and independent sets.  In contrast to an ordinary graph, in which each edge connects exactly two vertices, an edge (or, ``hyperedge'') can join any number of vertices in a hypergraph.  In game theory, hypergraphs are called simple games.  Hypergraphs have been used for problems in biology, chemistry, image processing, wireless networks, and more.  
In the current work, zeon (``nil-Clifford'') and ``idem-Clifford'' graph-theoretic methods are generalized to hypergraphs.  In particular, zeon and idem-Clifford methods are used to enumerate paths, trails, independent sets, cliques, and matchings in hypergraphs.  An approach for finding minimum hypergraph transversals is developed, and zeon formulations of some open hypergraph problems are presented.
keywords: Cycle, Game, Hypergraph, Path, Transversal, Zeon
\end{abstract}

\section{Introduction}
While graphs have proven to be useful models for many real-world problems, edges are limited to modeling pairwise relations.  Hypergraphs have proven to be useful models for problems where pairwise relations are not sufficient as hypergraphs allow for edges to contain more than two vertices and hence are able to model more complicated relationships among a collection of objects.  Hypergraphs have been used for problems in
biology~\cite{Klamt2009}, chemistry~\cite{Elena2001}, image processing~\cite{Bretto2004}, wireless networks~\cite{Zhang2018}, and more.  Many problems modeled by hypergraphs can be resolved by finding particular structures within a hypergraph.  

Quite recently, hypergraphs have been studied for their relevance in machine learning.  Neural networks are considered in \cite{Feng2019}, random walks and their applications to image segmentation are considered in \cite{Ducournau}.  Hypergraph learning is applied to social media networks in \cite{Fang}, while hypergraph partitioning is applied to document clustering in \cite{Hu}.  Biological applications include cancer outcome prediction~\cite{Hwang}.

This paper offers a new approach to revealing some structures in a hypergraph via zeon and idem-Clifford algebras.  Combinatorial properties of zeon algebras have proven to be useful for solving many graph-theoretic problems \cite{Staples2008}, \cite{Staples2019}, \cite{StaplesStellhorn2017}.  In previous works, zeons have been used to enumerate paths, cycles, matchings, and cliques in (ordinary) graphs.  We extend these results to the hypergraph setting and present a new approach to finding minimum transversals of a hypergraph, an important problem with many applications in computer science \cite{Eiter2002,Eiter1995}.  The approach lends itself to symbolic computation; examples herein have been computed using {\em Mathematica}.  

The paper is laid out as follows.  Necessary terminology regarding graphs and hypergraphs is given in Section \ref{prelims}, along with definitions of essential algebraic structures; i.e., zeon,  idem-Clifford, and generalized zeon algebras.  Section \ref{enumeration problems} discusses applications of zeon and idem-Clifford algebras to enumeration problems on graphs and hypergraphs.  In particular, we extend results for paths, cliques and independent sets, and matchings to the hypergraph setting.  We also provide an idem-Clifford approach to finding minimum hypergraph transversals in this section.  In Section \ref{zeon formulations}, some open hypergraph problems are reformulated in terms of zeon and idem-Clifford algebras.  The paper closes with concluding remarks and potential future work in Section \ref{conclusion}.

\section{Preliminaries}\label{prelims}

To begin, we provide a brief overview of the necessary graph and hypergraph terminology.  

\subsection{Basic Graph and Hypergraph Theory}

The following definitions can be found in \cite{Zhang2018}, but it should be noted that hypergraph terminology is far from standardized and there are many ways to generalize the various properties of graphs.

\begin{definition}
A {\em hypergraph} $H=(V,E)$ is a set $V$, whose elements are called {\em vertices}, together with a set $E=\{e_1,e_2,\dots,e_m\}$ of nonempty subsets\footnote{A hypergraph whose edges consist of vertex pairs is commonly called a {\em graph}. 
} of $V$ called {\em hyperedges}.   
\end{definition}

A hyperedge containing a single vertex is called a {\em loop}.  A hyperedge which is a subset of another hyperedge is said to be {\em included}.  A hypergraph containing no included hyperedges and no loops is said to be {\em simple}.  

A hypergraph is {\em finite} when $V$ and $E$ are finite sets.  The results that follow will not require the hypergraph to be simple, but all hypergraphs will be considered finite unless stated otherwise.

Two vertices $v_i,v_j \in V$ are {\em adjacent} if there is at least one hyperedge containing both $v_i$ and $v_j$.  If a hyperedge $e$ contains vertex $v$, we say that they are {\em incident} to each other.  The {\em degree} of a vertex is the number of hyperedges incident to the vertex.

A matrix is often useful to represent incidence relations in a hypergraph.

\begin{definition}
The {\em incidence matrix} of a hypergraph $H$ is the matrix $C$ whose rows represent vertices of $H$ and whose columns represent hyperedges of $H$ with entries defined by 
\begin{equation*}
C_{ij}= \begin{cases}
    1 &\text{  if  } v_i \in e_j \\
    0 &\text{otherwise.}
\end{cases}
\end{equation*}
\end{definition}

\begin{figure}
\centering
\begin{minipage}[b]{.45\textwidth}
\includegraphics[width=\textwidth]{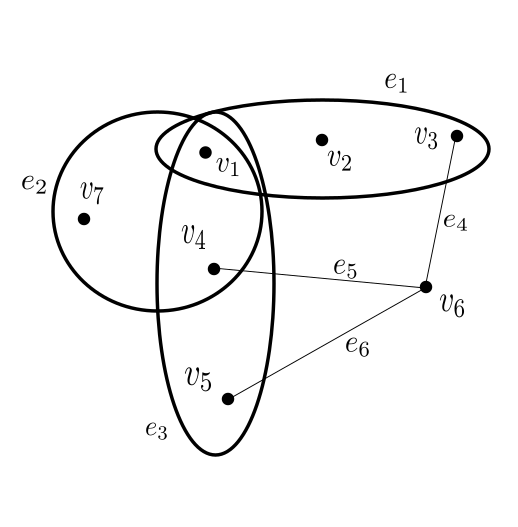}
\caption{A hypergraph $H$\label{first}}
\end{minipage}\hfill
\begin{minipage}[b]{.45\textwidth}
\centering
\small
\(
C = 
\begin{pmatrix}
    1 & 1 & 1 & 0 & 0 & 0 \\
    1 & 0 & 0 & 0 & 0 & 0 \\
    1 & 0 & 0 & 1 & 0 & 0 \\
    0 & 1 & 1 & 0 & 1 & 0 \\
    0 & 0 & 1 & 0 & 0 & 1 \\
    0 & 0 & 0 & 1 & 1 & 1 \\
    0 & 1 & 0 & 0 & 0 & 0 \\
\end{pmatrix}
\)
\caption{The incidence matrix of $H$}\label{second}
\end{minipage}
\end{figure}

It is often difficult to work with arbitrary hypergraphs given the differing sizes of hyperedges.  One may wish to force hyperedges to be a particular size.  A hypergraph is said to be {\em $r$-uniform} if each hyperedge in the hypergraph has cardinality $r$.  Note that ordinary simple graphs are $2$-uniform hypergraphs.

\subsection{Zeon and idem-Clifford Algebras}

In this section we define the algebraic structures we will use to model various hypergraph problems.  

\subsubsection{Zeon Algebras}
The {\em n-particle zeon algebra}, $\Z_n$, is the real abelian algebra generated by the set $\{\zeta_i : 1 \leq i \leq n\}$ along with the unit scalar $1=\zeta_{\emptyset}$ subject to the following multiplication rules:
\begin{align*}
\zeta_i\zeta_j &= \zeta_j\zeta_i \quad \text{for} \quad i \neq j,\\
\zeta_i^2 &= 0 \quad \text{for} \quad 1 \leq i \leq n.
\end{align*}

Let $[n]$ denote the $n$-set $\{1,2,\dots,n\}$ and denote the power set of $[n]$ by $2^{[n]}$. Let the {\em multi-index notation} $\zeta_I$  denote the product $$\zeta_I = \prod_{i \in I}\zeta_i$$ for any $I \in 2^{[n]}$.  The algebra $\Z_n$ is spanned by the collection $\{\zeta_I: I\in 2^{[n]}\}$ of {\em basis blades}.  Accordingly, each element $u \in \Z_n$ has a canonical expansion of the form $u = \sum_{I} u_I\zeta_I$,
where $u_I \in \mathbb{R}$ for each $I\in2^{[n]}$.

\begin{remark}
To simplify notation, generators will be denoted simply by $\zeta_i$ rather than $\zeta_{\{i\}}$.   This convention will be extended to all algebras introduced henceforth.
\end{remark}

It is often convenient to separate the scalar (grade-0) part of a zeon element from the rest of it.  Given $z\in\Z_n$, we write $\zre z=\langle z\rangle_0$ for the scalar part of $z$ and $\zdu z=z-\zre z$ for the dual part of $z$.

The {\em scalar sum} of $u\in\Z_n$ is defined to be the sum of scalar coefficients in $u$ with respect to the canonical basis; i.e., \begin{equation*}
\langle\langle u\rangle\rangle=\left<\left< \sum_I u_i\zeta_I \right>\right>=\sum_I u_I.
\end{equation*}

Zeon algebras have proven to be useful in enumeration problems on graphs where certain configurations are forbidden, such as in the enumeration of matchings and self-avoiding walks.  To perform such enumerations, one constructs algebraic representations of the graph using elements of a zeon algebra.  The nilpotent property of the generators can then be exploited to cancel unwanted configurations through multiplication.  It is also often useful to consider nilpotent zeons other than generators.  

\begin{definition}
Given a nilpotent element $u\in\Z_n$, the {\em degree of nilpotency} of $u$ is the smallest positive integer $\kappa(u)$ such that $u^{\kappa(u)}=0$.  \end{definition}

\begin{definition}
For a zeon $u\ne 0$, it is useful to define the {\em minimal grade} of $u$ by \begin{equation*}
\natural u =\begin{cases}
\min\left\{k\in \mathbb{N}: \langle \zdu u\rangle_k\ne 0\right\}& \zdu u\ne 0,\\
0&u=\zre u.
\end{cases}
\end{equation*}
Note that $\natural u=0$ if and only if $u$ is trivial.
\end{definition}

\subsubsection{Generalized Zeon Algebras}

Generalized zeons were introduced in \cite{SchottStaples2017}.   Observing that the sum of $k$ zeon generators is nilpotent of index $k+1$, e.g. $(\zeta_{1}+\zeta_{2}+\cdots+\zeta_{k})^k=k!\zeta_{\{1,2,\ldots, k\}}$ and $(\zeta_{1}+\zeta_{2}+\cdots+\zeta_{k})^{k+1}=0$, it is possible to construct nilpotent elements having any finite index of nilpotency within a zeon algebra of sufficiently high dimension.   

\begin{definition}
For positive integer $n$, let $\mathbf{s}=(s_1,\ldots, s_n)\in{\mathbb{N}}^n$ be an $n$-tuple of positive integers.  Then, the {\em zeon algebra of signature} $\mathbf{s}$ (or {\em $\mathbf{s}$-zeon algebra}), denoted here by $\Z_{\mathbf{s}}$, is the real associative algebra generated by the collection $\displaystyle\left\{\underset{s_i}{\nu_i}{}: 1\le i\le n\right\}$ along with the scalar $1=\nu_\varnothing$, subject to the following multiplication rules: \begin{itemize}
\item ${\underset{s_i}{\nu_i}{}}{\underset{s_j}{\nu_j}{}}={\underset{s_j}{\nu_j}{}}{\underset{s_i}{\nu_i}{}}$ for all $1\le i,j\le n$;
\item ${\underset{s_i}{\nu_i}{}}^{k}=0$ if and only if $k\ge s_i$;
\item ${\underset{s_i}{\nu_i}{}}{\underset{s_j}{\nu_j}{}}=0$ if and only $i=j$ and $s_i=2$.
\end{itemize}
\end{definition}

\begin{example}
The algebra $\Z_{(2,3,5)}$ has three generators: $\displaystyle\left\{\underset{2}{\nu_1}{}, \underset{3}{\nu_2}{}, \underset{5}{\nu_3}{} \right\}$.  These generators are pairwise commutative and each is nilpotent.  In particular,  $\underset{2}{\nu_1}{}^2=0$, $\underset{3}{\nu_2}{}^3=0$, and $\underset{5}{\nu_3}{}^5=0$.  General products are computed using distributivity and associativity.  For example, \begin{eqnarray*}
\left(\underset{3}{\nu_2}{}+2\underset{2}{\nu_1}{}\right)^2&=&\underset{3}{\nu_2}{}^2+4\underset{2}{\nu_1}{}\underset{3}{\nu_2}{},\\
\left(1-\underset{2}{\nu_1}{}+\underset{5}{\nu_3}{}^2\right)\underset{5}{\nu_3}{}^3&=&\underset{5}{\nu_3}{}^3-\underset{2}{\nu_1}\underset{5}{\nu_3}{}^3.
\end{eqnarray*}

\end{example}

\subsubsection{Idem-Clifford Algebras}

We can also consider algebras whose generators are idempotent rather than nilpotent.  While combinatorial properties of these algebras have not been explored as deeply, we will later show how the idempotent property of the generators can be used to avoid redundant information when enumerating certain graph and hypergraph structures.  These algebras have also been used in Boolean satisfiability problems in \cite{DavisStaples2019} and are defined as follows.

The {\em $n$-generator idem-Clifford algebra}, $\I_n$, is the real abelian algebra generated by the set $\{\varepsilon_i \mid 1 \leq i \leq n \}$ along with the scalar $1 = \varepsilon_{\emptyset}$ subject to the following multiplication rules:

\begin{itemize}
    \item[] \begin{center} $\varepsilon_i\varepsilon_j = \varepsilon_j\varepsilon_i \quad \text{for} \quad i \neq j$, and\end{center}
    \item[] \begin{center} ${\varepsilon_i}^2 = \varepsilon_i \quad \text{for} \quad 1 \leq i \leq n$.\end{center}
\end{itemize}
With multi-index notation in mind, each element $u \in \I_n$ has the expansion $$u = \sum_{I \in 2^{[n]}} u_I\varepsilon_I$$ where $u_I \in \mathbb{R}$ for each $I$.

\begin{example}
In $\I_6$, one finds that
\begin{eqnarray*}
(\varepsilon_{2}-4\,\varepsilon_{6})^2&=&\varepsilon_{2}-8\,\varepsilon_{\{2,6\}}+16\,\varepsilon_{6},\\
(3\,\varepsilon_{\{1,2\}}+\varepsilon_{3})(\varepsilon_{1}-2\,\varepsilon_{4})&=&3\,\varepsilon_{\{1,2\}}-6\,\varepsilon_{\{1,2,4\}}+\varepsilon_{\{1,3\}}-2\,\varepsilon_{\{3,4\}}.
\end{eqnarray*} 
\end{example}

\begin{remark}
The algebras $\Z_n$ and $\I_n$ have often been denoted by ${\mathcal{C}\ell_n}^{\rm nil}$ and ${\mathcal{C}\ell_n}^{\rm idem}$, respectively, to emphasize their relationship to Clifford algebras.  In particular, these algebras can be constructed as subalgebras of Clifford algebras of appropriate signature~\cite{DavisStaples2019,SchottStaples2017}.
\end{remark}

\subsection{Infinite-Dimensional Algebras}

Observing that $\Z_n$ is a subalgebra of $\Z_N$ for any $n\le N$ and that $\I_m$ is a subalgebra of $\I_M$ for $m\le M$, notation will be greatly simplified by referring only to algebras $\Z$ and $\I$ with the assumption that the number of generators is sufficient for matters at hand.  The notation $\Z_\mathbf{s}$ will denote a generalized zeon algebra of signature $\mathbf{s}$.

Since all graphs and hypergraphs considered here are finite, it is clear that the constructions of associated zeon, generalized zeon, and idem-Clifford elements can be achieved within algebras of sufficiently high dimension.

\section{Hypergraph Enumeration Problems}\label{enumeration problems}

Graph enumeration problems are problems involving counting structures in graphs such as various types of walks, matchings, cliques, independent sets, and more.  Zeon algebras have been used to count self-avoiding walks on graphs, independent sets and cliques, and matchings.  In this section we extend these results to the hypergraph setting.  An idem-Clifford approach to enumerating minimum hypergraph transversals (and as a consequence, minimum graph coverings) is also presented.

\subsection{Walks on Hypergraphs}
A {\em $k$-walk} in a graph $G$ is a sequence of vertices $(v_0,\dots,v_k)$ such that there exists an edge $\{v_j,v_{j+1}\} \in E$ for each $0 \leq j \leq k-1$.  We call $v_0$ the {\em initial vertex} and $v_k$ the {\em terminal vertex}.  A {\em self-avoiding walk} or {\em path} is a walk in which each vertex appears at most once.  A {\em closed} $k$-walk is a $k$-walk in which $v_0=v_k$.  A {\em k-cycle} is a closed $k$-path where we allow $v_0$ to be repeated exactly once as the terminal vertex.  A {\em k-trail} is a $k$-walk in which no edge is repeated.

Note that in the hypergraph case, vertices can be adjacent in any number of hyperedges so the definition of a hypergraph walk must also include information about the hyperedge set.  

\begin{definition}
A {\em $k$-walk} in a hypergraph $H=(V,E)$ is a sequence $\break\hfill$ $(v_0,e_1,v_1,\dots,v_{k-1},e_k,v_k)$ such that $v_0,\dots,v_k \in V$ and $e_1,\dots,e_k \in E$ have the property that $v_{j-1}$ is adjacent to $v_{j}$ in hyperedge $e_j$ for each $1 \leq j \leq k$.  
\end{definition}

The notions of paths, closed walks, cycles, and trails are the same as those in ordinary graphs.  

One traditional method for enumerating walks on graphs is to consider powers of the graph's adjacency matrix.  However, powers of the adjacency matrix fail to distinguish between walks and paths in a graph.  Nilpotent adjacency matrices were developed in \cite{Staples2008} in order to be able to make this distinction and enumerate $k$-paths in graphs. 

Note that while an adjacency matrix for a hypergraph could be constructed in a similar fashion as for ordinary graphs, such a matrix does not correspond to a unique hypergraph.  This is because the adjacency matrix does not account for the fact that vertices may be adjacent in multiple edges in the hypergraph case.  Due to this, many notions of a hypergraph adjacency matrix have been considered, as discussed in \cite{Ouvard2018}.  

\subsection{The hypergraph nilpotent adjacency matrix}

We aim to construct a matrix that distinguishes between walks and paths and also retains information about the hyperedge set of a walk.  To do this, we will first consider the bipartite graph representation of a hypergraph.  This representation is constructed as follows: for a hypergraph $H=(V,E)$, let $B_H$ be a bipartite graph whose vertex sets are $V$ and $E$.  Vertices $v$ and $e$ are adjacent in $B_H$ when $e$ is a hyperedge incident with $v$ in $H$.  The bipartite representation of the hypergraph $H$ from Figure \ref{first} is shown in Figure \ref{bipartite rep}.

\begin{figure}[h]
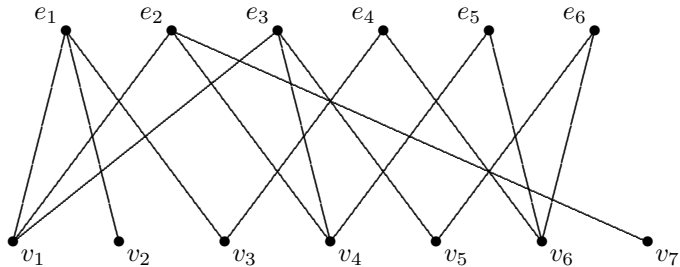

$$    \begindc{\undigraph}[200]
    \obj(14,1)[v7]{$v_7$}[\southeast]
    \obj(12,1)[v6]{$v_6$}[\southeast]
    \obj(10,1)[v5]{$v_5$}[\southeast]
    \obj(8,1)[v4]{$v_4$}[\southeast]
    \obj(6,1)[v3]{$v_3$}[\southeast]
    \obj(4,1)[v2]{$v_2$}[\southeast]
    \obj(2,1)[v1]{$v_1$}[\southeast]
    \obj(13,5)[e6]{$e_6$}[\northwest]
    \obj(11,5)[e5]{$e_5$}[\northwest]
 \obj(9,5)[e4]{$e_4$}[\northwest]
 \obj(7,5)[e3]{$e_3$}[\northwest]
 \obj(5,5)[e2]{$e_2$}[\northwest]
 \obj(3,5)[e1]{$e_1$}[\northwest]
 \mor{v1}{e1}{}
    \mor{v1}{e2}{}
    \mor{v1}{e3}{}
    \mor{v2}{e1}{}
    \mor{v3}{e1}{}
    \mor{v3}{e4}{}
    \mor{v4}{e2}{}
    \mor{v4}{e3}{}
    \mor{v4}{e5}{}
    \mor{v5}{e3}{}
    \mor{v5}{e6}{}
    \mor{v6}{e4}{}
    \mor{v6}{e5}{}
    \mor{v6}{e6}{}
    \mor{v7}{e2}{}
   \enddc
$$\caption{The bipartite representation of $H$.}\label{bipartite rep}
\end{figure}

One could apply the nilpotent adjacency matrix to this bipartite graph, however it is very restrictive in the sense that it allows for neither vertices nor hyperedges to be repeated in the corresponding hypergraph walk.  We can modify the nilpotent adjacency matrix to allow for repeated hyperedges by instead labeling their corresponding vertices in $B_H$ with idem-Clifford generators.  This allows one to enumerate walks in $H$ with no repeated vertices while still maintaining information about the hyperedge set in the walk. Let $H$ be a hypergraph with $n$ vertices and $m$ hyperedges.  Let $\{\varepsilon_1,\varepsilon_2,\dots,\varepsilon_m\}$ be generators of $\I_m$. We define the  $(n+m) \times (n+m)$ matrix $\mathcal{B}=(b_{ij})$ by

\begin{equation*}
b_{ij} = 
\begin{cases}
    \varepsilon_{j-n} & \text{  if  } 1 \leq i \leq n, \text{  } n+1 \leq j \leq n+m,  \text{ and $v_i,e_{j-n}$ incident in $H$}\\
    \zeta_j & \text{  if  } n+1 \leq i \leq n+m, \text{  }1 \leq j \leq n, \text{ and $v_{j},e_{i-n}$ incident in $H$}\\
    0 & \text{otherwise.}
\end{cases}
\end{equation*}

If one wishes to look for a $k$-walk in $H$, one could look for a corresponding $2k$-walk in $B_H$ by taking even powers of $\mathcal{B}$.  However, it is possible to reduce the size of $\mathcal{B}$ while still maintaining all vertex and hyperedge information.  First note that $\mathcal{B}$ is a block matrix consisting of zero blocks and modified incidence matrices of $H$.  Specifically,  if $C$ is the $n \times m$ incidence matrix of $H$, then $\mathcal{B}$ is of the form 

$$ \mathcal{B} = 
\begin{pmatrix}
    0
  & \vline & X
 \\
\hline
  Z & \vline &
  0
\end{pmatrix}
$$

where $X$ is the $n \times m$ matrix \begin{equation*}
X = C 
  \begin{pmatrix}
    \varepsilon_{1} & 0 & \dots & 0 \\
    0 & \varepsilon_{2} & \dots & 0 \\
    \vdots & \vdots & \ddots & \vdots \\
    0 & 0 & \dots & \varepsilon_{m}
  \end{pmatrix}
  \end{equation*}
and $Z$ is the $m \times n$ matrix \begin{equation*}
Z = C^{T} 
  \begin{pmatrix}
    \zeta_{1} & 0 & \dots & 0 \\
    0 & \zeta_{2} & \dots & 0 \\
    \vdots & \vdots & \ddots & \vdots \\
    0 & 0 & \dots & \zeta_{n}
  \end{pmatrix}.
\end{equation*}
Now note that when looking for $k$-walks in $H$ using $\mathcal{B}$, one is only interesting in taking even powers of $\mathcal{B}$.  In particular, we note that $\mathcal{B}^2$ is a block matrix of the form 
$$\mathcal{B}^2 = \begin{pmatrix}
         XZ
  & \vline & 0
 \\
\hline
  0 & \vline &
  ZX
    \end{pmatrix}.$$

Due to the block structure, higher-ordered even powers of $\mathcal{B}$ will have the same block structure.  As we will soon show, all information regarding $k$-paths in $H$ can be recovered from taking powers of the $n \times n$ matrix $XZ$.  In light of this, we propose the following definition of the nilpotent hypergraph adjacency matrix.

\begin{definition}
Let $H=(V,E)$ be a hypergraph on $n$ vertices and $m$ hyperedges with $n \times m$ incidence matrix $C$.  Let $X$ be the $n \times m$ matrix given by $$X = C  \begin{pmatrix}
    \varepsilon_{1} & 0 & \dots & 0 \\
    0 & \varepsilon_{2} & \dots & 0 \\
    \vdots & \vdots & \ddots & \vdots \\
    0 & 0 & \dots & \varepsilon_{m}
  \end{pmatrix}$$ and let $Z$ be the $m \times n$ matrix given by \begin{equation*}
  Z =  C^{T} \begin{pmatrix}
    \zeta_{1} & 0 & \dots & 0 \\
    0 & \zeta_{2} & \dots & 0 \\
    \vdots & \vdots & \ddots & \vdots \\
    0 & 0 & \dots & \zeta_{n}
  \end{pmatrix}.
  \end{equation*}  The {\em nilpotent adjacency matrix of the hypergraph} $H$ is defined to be the $n \times n$ matrix $\Omega = XZ$, whose entries are elements of $\Z_n\otimes\I_m$.
\end{definition}

\begin{example}
For the hypergraph $H$ in Figure \ref{first}, we see that
\begin{equation*}
X = \begin{pmatrix}
\varepsilon_{1} & \varepsilon_{2} & \varepsilon_{3} & 0 & 0 & 0 \\
\varepsilon_{1} & 0 & 0 & 0 & 0 & 0 \\
\varepsilon_{1} & 0 & 0 & \varepsilon_4 & 0 & 0 \\
0 & \varepsilon_{2} & \varepsilon_{3} & 0 & \varepsilon_{5} & 0 \\
0 & 0 & \varepsilon_{3} & 0 & 0 & \varepsilon_6 \\
0 & 0 & 0 & \varepsilon_{4} & \varepsilon_{5} & \varepsilon_{6} \\
0 & \varepsilon_{2} & 0 & 0 & 0 & 0 
\end{pmatrix}
\end{equation*}
and
\begin{equation*}
Z =\begin{pmatrix}
\zeta_{1} & \zeta_{2} & \zeta_{3} & 0 & 0 & 0 & 0 \\
\zeta_{1} & 0 & 0 & \zeta_{4} & 0 & 0 & \zeta_{7} \\
\zeta_{1} & 0 & 0 & \zeta_{4} & \zeta_{5} & 0 & 0 \\
0 & 0 & \zeta_{3} & 0 & 0 & \zeta_{6} & 0 \\
0 & 0 & 0 & \zeta_{4} & 0 & \zeta_{6} & 0 \\
0 & 0 & 0 & 0 & \zeta_{5} & \zeta_{6} & 0 
\end{pmatrix}.
\end{equation*}\end{example}
The nilpotent hypergraph adjacency matrix of $H$ is the $6\times 9$ matrix \begin{equation*}\Omega=XZ=(\Omega_L\mid \Omega_C \mid \Omega_R),
\end{equation*}
 where $\Omega_L$, $\Omega_C$, and $\Omega_R$ are the following $6\times 3$, $6\times 1$, and $6\times 3$ matrices, respectively:
\begin{align*} \Omega_L&=   
    \begin{pmatrix}
    \zeta_{1}(\varepsilon_{1}+\varepsilon_{2}+\varepsilon_{3}) & \zeta_{2}\,\varepsilon_{1} & \zeta_{3}\,\varepsilon_{1} \\
    \zeta_{1}\,\varepsilon_{1} & \zeta_{2}\,\varepsilon_{1} & \zeta_{3}\,\varepsilon_{1}  \\
    \zeta_{1}\,\varepsilon_{1} & \zeta_{2}\,\varepsilon_{1} & \zeta_{3}(\varepsilon_{1}+\varepsilon_{4}) \\
    \zeta_{1}(\varepsilon_{2}+\varepsilon_{3}) & 0 & 0  \\
    \zeta_{1}\,\varepsilon_{3} & 0 & 0 \\
    0 & 0 & \zeta_{3}\,\varepsilon_{4} \\
    \zeta_{1}\,\varepsilon_{2} & 0 & 0
    \end{pmatrix} 
\end{align*}
\begin{align*} \Omega_C&=   
    \begin{pmatrix}
 \zeta_{4}(\varepsilon_{2}+\varepsilon_{3})  \\
 0 \\
0  \\
  \zeta_{4}(\varepsilon_{2}+\varepsilon_{3}+\varepsilon_{5}) 
  \\
  \zeta_{4}\,\varepsilon_{3}  \\
  \zeta_{4}\,\varepsilon_{5}  \\
  \zeta_{4}\,\varepsilon_{2}  
    \end{pmatrix}
\end{align*}

\begin{align*} \Omega_R&=   
    \begin{pmatrix}
 \zeta_{5}\,\varepsilon_{3} & 0 & \zeta_{7}\,\varepsilon_{2} \\
 0 & 0 & 0 \\
 0 & \zeta_{6}\,\varepsilon_{4} & 0 \\
 \zeta_{5}\,\varepsilon_{3} & \zeta_{6}\varepsilon_{5} & \zeta_{7}\,\varepsilon_{2} \\
  \zeta_{5}(\varepsilon_{3}+\varepsilon_{6}) & \zeta_{6}\,\varepsilon_{6} & 0 \\
 \zeta_{5}\,\varepsilon_{6} & \zeta_{6}(\varepsilon_{4}+\varepsilon_{5}+\varepsilon_{6}) & 0 \\
 0 & 0 & \zeta_{7}\,\varepsilon_{2} 
    \end{pmatrix}.
\end{align*}

Define $\langle\zeta_{i}|$  to denote the row vector with zeon generator $\zeta_{i}$ in the $i$th column and zeros everywhere else.  Define $|j\rangle$ to denote the column vector with $1$ in the $j$th row with zeros everywhere else.  We can now enumerate $k$-paths in a hypergraph as formalized in the following result.   

\begin{theorem}
Let $H=(V,E)$ be a hypergraph with $n$ vertices and $m$ hyperedges and let $\Omega$ be the nilpotent hypergraph adjacency matrix of $H$. Then for $k \in \mathbb{N}$ and $1 \leq i \neq j \leq n$, we have 
\begin{equation*}\langle \zeta_{i}|\Omega^k|j\rangle = \sum_{\substack{I \subseteq V \\ |I|=k+1}} \omega_I\zeta_I\varepsilon_J,
\end{equation*}
where $\omega_I$ is the number of $k$-paths from $v_i$ to $v_j$ in $H$ on vertex set $I$ and hyperedge set $J$.  Further, when $i=j$ and $k \geq 2$, we have \begin{equation*}
\langle i | \Omega^k | i \rangle = \sum_{\substack{I \subseteq V \\ |I| = k}} \omega_I \zeta_I\varepsilon_J,
\end{equation*} where $\omega_I$ denotes the number of $k$-cycles in $H$ based at vertex $v_i$ on vertex set $I$ and hyperedge set $J$.
\end{theorem}

\begin{proof}
Proof is by induction on $k$.  By the construction of $\Omega$, nonzero entries of $\Omega$ are of the form $$\Omega_{ij}=\zeta_{j}\sum_{e_{\ell},v_i \text{incident in $H$}}\varepsilon_{\ell}$$ and so
    $$\langle \zeta_{\{i\}} \mid \Omega \mid j \rangle = \zeta_{\{i,j\}}\sum_{e_{\ell},v_i \text{incident in $H$}}\varepsilon_{\ell}$$
when $v_i$ and $v_j$ are adjacent in $H$ and $0$ otherwise.
Hence the claim holds for $k=1$.  Now suppose the claim holds for some $k \geq 1$.  Write $\Omega_{ij}^k = \langle i \mid \Omega^k \mid j \rangle$.  It follows from the inductive hypothesis that $\displaystyle\Omega_{ij}^k = \sum_{|I|=k} \omega_I\zeta_I\varepsilon_J$ where $\omega_I$ is the number of $k$-walks from $v_i$ to $v_j$ on vertex set $I$ and hyperedge set $J$ which do not revisit any vertex except for possibly $v_i$ exactly once.  Now writing $\Omega^{k+1}=\Omega^k \Omega$ implies that
\begin{align*}
 {\Omega_{ij}}^{k+1} & = \sum_{r=1}^n {\Omega_{ir}}^k \Omega_{rj} \\
         & = \sum_{r=1}^n \left( \sum_{|I|=k} \omega_I\zeta_I \varepsilon_J \right) \Omega_{rj}
\end{align*}
where $\omega_I$ is the number of $k$-walks from $v_i$ to $v_r$ on vertex set $I$ and hyperedge set $J$ which do not revisit any vertex except for possibly $v_i$ exactly once.  Since $\Omega_{rj} = \zeta_{j}\displaystyle\sum_{e_{\ell},v_r \text{incident in $H$}} \varepsilon_{\ell}$ when $v_r$ and $v_j$ are adjacent and $0$ otherwise, we have that ${\Omega_{ij}}^{k+1} = \displaystyle\sum_{|I|=k+1}\omega_I\zeta_I \varepsilon_J$ where $\omega_I$ is the number of $(k+1)$-walks from $v_i$ to $v_j$ which do not revisit any vertex except possibly $v_i$ exactly once.  Multiplication by $\zeta_{\{i\}}$ cancels walks which revisit $v_i$ and so $\langle \zeta_{\{i\}} \mid \Omega^{k+1} \mid j \rangle = \displaystyle\sum_{|I|=k+2}\omega_I\zeta_I \varepsilon_J$ where $\omega_I$ is the number of $(k+1)$-paths from $v_i$ to $v_j$ and the claim holds when $i \neq j$.  When $i = j$, the last step of the walk revisits $v_i$ and no cancellation is necessary.  In this case, $\langle i \mid \Omega^{k+1} \mid j \rangle = \displaystyle\sum_{|I|=k+1}\omega_I\zeta_I \varepsilon_J$ where $\omega_I$ is the number of $k$-cycles based at $v_i$ on vertex set $I$ and hyperedge set $J$.  Hence we have the result by induction.
\end{proof}

\begin{example}
For the hypergraph in Figure \ref{first}, we see that 
\begin{align*}
\langle \zeta_{3} \mid \Omega^3 \mid 4 \rangle &= \zeta_{\{1,2,3,4\}}\varepsilon_{\{1,2\}}+\zeta_{\{1,2,3,4\}}\varepsilon_{\{1,3\}}+\zeta_{\{1,3,4,5\}}\varepsilon_{\{1,3\}}\\&+\zeta_{\{1,3,4,7\}}\varepsilon_{\{1,2\}}+\zeta_{\{3,4,5,6\}}\varepsilon_{\{3,4,6\}},
\end{align*} which shows that there are five $3$-paths from $v_3$ to $v_4$ in $H$.
\end{example}

We could also use the nilpotent hypergraph adjacency matrix to enumerate $k$-trails in a hypergraph by interchanging the roles of zeon and idempotent variable labels, i.e. labeling the vertices in $B_H$ corresponding to vertices in $H$ with idempotent variables and labeling the vertices corresponding to hyperedges with zeon generators.  

\subsection{Cliques and Independent Sets}
In this section, we will discuss applications of zeons to finding cliques and independent sets in graphs and hypergraphs.  Recall that a {\em clique} in a graph $G$ is a subgraph of $G$ which is isomorphic to a complete graph.  Equivalently, a clique in a graph is a subset of vertices which are pairwise adjacent. An {\em independent set} in a graph $G$ is a subset of vertices which are pairwise non-adjacent.  It is not hard to see that cliques in a graph $G$ form independent sets in the complement of $G$.  

In the case of an ordinary graph $G=(V,E)$, one can find cliques in $G$ by finding independent sets in $G'$, the complement of $G$.  Independent sets of $G'$ can be found by constructing a polynomial with zeon coefficients whose powers reveal independent sets in a graph \cite{Staples2019}. 
 To construct this polynomial, called the ``zeon clique representation'' of $G$, we will first need to make any isolated vertices of $G'$ self-adjacent by adding loops to the edge set of $G'$.  More specifically, we create a graph $G''=(V,E'')$ with a new edge set given by $E'' = E' \cup L$ where $L = \{\{v,v\} \mid v \text{ isolated in } G'\}$.   Let $\varphi: E'' \to \{\zeta_1,\zeta_2,\dots,\zeta_{|E''|}\}$ be a labeling of these edges and let $\lambda: V \to \{\varepsilon_1, \varepsilon_2, \ldots, \varepsilon_n\}$ (the generators of $\I_n$), be a labeling of $V$.  Also let $\psi(v)=\displaystyle\prod_{e\,\text{\rm incident to}\,v}\varphi(e)$ be the product of incident edge labels for each vertex.  Finally, the {\em zeon clique representation of $G$} is defined by \begin{equation}\label{zeon clique rep}
\Phi_{G'} = \sum_{v \in V}\psi(v)\lambda(v).
\end{equation}

\begin{theorem}\label{ind sets}
Let $G$ be a graph on $n$ vertices and let $\Phi_{G'}$ be the zeon clique representation of $G$.  Then for $k \in \mathbb{N}$, $$\Phi_{G'}^k=k!\sum_{I \subseteq V, \mid I \mid = k}\zeta_{N(I)}x_I$$ where $N(I)$ represents the edges of $G''$ which are incident with vertices in $I$ and $$x_I = \prod_{v \in I}x_v$$ is a monomial whose index set $I$ represents the vertices of an independent set of size $k$ in $G'$ and equivalently a clique in $G$.
\end{theorem}

We now aim to find cliques and independent sets in hypergraphs, but a few difficulties arise in the case of hypergraphs.  The first difficulty is that in ordinary graphs, an independent set can be viewed as a collection of pairwise non-adjacent vertices or equivalently as a collection of vertices which does not contain any edge of the graph. These are not equivalent properties in hypergraphs due to the arbitrary size of hyperedges.  This non-equivalence inspires the following definitions seen in \cite{Halldorsson2009}.

\begin{definition}
A {\em weak independent set} in a hypergraph $H=(V,E)$ is a subset of $V$ which does not contain any hyperedge of $H$.  A {\em strong independent set} in a hypergraph $H=(V,E)$ is a subset of $V$ which intersects any hyperedge in $E$ in at most one vertex.
\end{definition}

Strong independent sets can be viewed as a special case of $k$-independent sets defined in \cite{Cutler2013}.  Specifically, strong independent sets are $1$-independent sets described in the following definition. Also note that independent sets in an ordinary graph are strong independent sets in $2$-uniform hypergraphs.

\begin{definition}
A {\em $k$-independent set} in a hypergraph $H=(V,E)$ is a collection of vertices $I$ such that $\mid I \cap e \mid \leq k$ for each $e \in E$.
\end{definition}

The second difficulty arises when trying to define cliques and complements in the hypergraph setting in a way that gives a correspondence between cliques and independent sets.  A natural approach may be to let a hypergraph clique be a copy of a complete hypergraph, i.e. a set of vertices having all possible hyperedges, and defining the complement of a hypergraph $H=(V,E)$ to be the hypergraph $H'=(V,E')$ where $E' = \mathcal{P}(V) \setminus E$.  Under this approach, the correspondence between cliques and weak independent sets is achieved.  However, the zeon clique representation is not a sufficient tool for revealing cliques in this case.  Consider a hypergraph which contains a clique and also contains a hyperedge disjoint from the clique.  The complement of such a hypergraph would have a hyperedge containing all vertices in the clique as well as the vertices in the disjoint hyperedge.  Following the construction of the zeon clique representation would label this hyperedge with a zeon generator and so the vertices of the clique in the original hypergraph cannot appear together in a power of the zeon clique representation.  Other definitions for hypergraph cliques and complements or some alternative labeling of hyperedges may resolve this issue, but it remains open. Given these difficulties, our results will be limited to finding weak and strong independent sets in hypergraphs.

To find weak independent sets in a hypergraph $H$, we construct an element $\Phi_H\in\Z_s\otimes\I$ for suitable choice of zeon signature $\mathbf{s}$.  In this case,  each hyperedge is identified with a zeon generator whose degree of nilpotency is equal to the cardinality of the hyperedge.  Doing this prevents us from appending loops to isolated vertices as we did in the case of ordinary graphs since they would not be distinguishable from singleton hyperedges.  Isolated vertices may appear in weak independent sets, but vertices forming singleton hyperedges may not.  With this in mind, we consider the case of hypergraphs containing no isolated vertices.

\begin{definition}
Let $H=(V,E)$ be a hypergraph with $m$ hyperedges and no isolated vertices. Assign a distinct generalized zeon element $u$ to each hyperedge $e$ such that the degree of nilpotency of $u$ is equal to $|e|$, and let $\mathbf{s}$ denote the signature of the corresponding generalized zeon algebra.  Let $\varphi(e_j)=u_j$ denote the label of hyperedge $e_j$, $1 \leq j \leq m$.  Let $\lambda: V \to \{\varepsilon_1, \varepsilon_2, \ldots, \varepsilon_n\}$, where $ \{\varepsilon_1, \varepsilon_2, \ldots, \varepsilon_n\}$ are generators of $\I_n$, be a labeling of $V$ and let $\psi(v)=\displaystyle\prod_{e\,\text{\rm incident to}\,v}\varphi(e)$ be the product of incident edge labels for each vertex.   The {\em zeon weak independent set representation} of $H$ is defined to be the $\Z_{\mathbf{s}}\otimes\I$ element \begin{equation*}
\Phi_H = \sum_{v \in V}\psi(v)\lambda(v).
\end{equation*}
\end{definition}

\begin{example}
Let $u_1$, $u_2$, and $u_3$ be distinct generalized zeons with degree of nilpotency 3 and let $u_4,u_5$, and $u_6$ be distinct zeons with degree of nilpotency 2.  
For $i=1, 2,\ldots, 6$, set $\varphi(e_i)=u_i$.  The zeon weak independent set representation of the hypergraph in Figure \ref{first} is then
\begin{align}
    \Phi_H &= \underset{3}{\nu _1}{}\underset{3}{\nu _2}{}\underset{3}{\nu _3}{}\,\varepsilon_{1}+\underset{3}{\nu _1}{}\,\varepsilon_{2}+\underset{3}{\nu _1}{}\underset{2}{\nu _4}{}\,\varepsilon_{3}+\underset{3}{\nu _2}{}\underset{3}{\nu _3}{}\underset{2}{\nu _5}{}\,\varepsilon_{4}\nonumber\\
    &\hspace{10pt}+\underset{3}{\nu _3}{}\underset{2}{\nu _6}{}\,\varepsilon_{5}+\underset{2}{\nu _4}{}\underset{2}{\nu _5}{}\underset{2}{\nu _6}{}\,\varepsilon_{6}+\underset{3}{\nu _2}{}\,\varepsilon_{7}.\label{weak ind set rep}
\end{align}
\end{example}

It is important to note that the zeon construction for hypergraphs differs from the ordinary graph construction in a significant way.  For the case of ordinary graphs, all edge labels have degree of nilpotency equal to 2.  Consequently, considering the multinomial expansion 
\begin{align*}
{\Phi_G}^k & = \left(\sum_{v \in V}\psi(v)\lambda(v)\right)^k \\
         & = \sum_{t_1+t_2+\dots+t_n=k} \binom{k}{t_1,t_2,\dots,t_n}\prod_{i=1}^n(\psi(v_i))^{t_i}\lambda(v_i)^{t_i},
\end{align*}
it follows that the only nonzero terms correspond to $n$-tuples $(t_1,t_2,\dots,t_n) \in \{0,1\}^n$.  As a result, each index set appearing in $\Phi_G^k$ must be of size $k$. However, the $t_i$ may be greater than one in the expansion of the hypergraph weak independent set representation.  This results in a product of fewer than $k$ of the $\varepsilon_i$'s and as so terms may appear in ${\Phi_H}^k$ whose index sets contain less than $k$ elements.  Note however that the sum in the multinomial expansion of ${\Phi_H}^k$ is taken over all possible choice of $(t_1,t_2,\dots,t_n)$ such that $t_1+t_2+\dots+t_n=k$.  In particular, we have a term for each $(t_1.t_2,\dots,t_n) \in \{0,1\}^n$ where $t_1+t_2+\dots+t_n=k$.  These terms represent vertex sets of size $k$ in a hypergraph $H$ and have nonzero coefficients precisely when they represent a weak independent set by construction of $\psi(v)$.  As a result, $\Phi_H$ reveals all weak independent sets of size $k$ in $H$ and may include terms that represent some, but not all, weak independent sets of size less than $k$.

\begin{proposition}\label{weak independent set enumeration}
Let $H=(V,E)$ be a hypergraph with no isolated vertices and let $\Phi_H$ be the zeon weak independent set representation of $H$.  Then for $k \in \mathbb{N}$, \begin{equation*}
{\Phi_H}^k = \sum_{I \subseteq V, \mid I \mid \leq k}\alpha_I u_J \varepsilon_I
\end{equation*}
where $u_J$ is the product of hyperedge labels incident to vertices in $I$ and \begin{equation*}
\varepsilon_I = \prod_{i \in I}\varepsilon_i
\end{equation*} 
is a basis blade of $\I_n$ whose index set represents an independent set of size $|I| \leq k$ in $H$.  In particular, if $H$ contains a weak independent set of size $k$, then it is represented by the index set of some $\varepsilon_I$ in ${\Phi_H}^k$.
\end{proposition}

\begin{proof}
  Terms in ${\Phi_H}^k$ can only be nonzero when $I \subseteq V$ is a collection of vertices not containing a hyperedge of $H$.  Otherwise, the edge label the contained hyperedge $e$ would appear at least $|e|$ times in $u_J$ and hence be zero.
\end{proof}

\begin{example}
    In terms of generators of $\Z_{(3,3,3,2,2,2)}$, the generalized zeons $u_1, \ldots, u_6$ associated with the hypergraph $H$ of Figure \ref{first} are as follows: 
\begin{center}
\begin{tabular}{l c r}
$u_1=\underset{3}{\nu _1}{}$ & $u_2=\underset{3}{\nu _2}{}$& $u_3=\underset{3}{\nu _3}{}$\\
$u_4=\underset{2}{\nu _4}{}$ & $u_5=\underset{2}{\nu _5}{}$ & $u_6=\underset{2}{\nu _6}{}$.
\end{tabular}
\end{center}
Letting $\Phi_H\in\Z_{(3,3,3,2,2,2)}\otimes \I_7$ denote the zeon weak independent set representation of $H$ as defined in \eqref{weak ind set rep},  direct computation via {\em Mathematica} yields the following:
\begin{align*}
    {\Phi_H}^5 &= 30 \underset{3}{\nu _1}{}^2 \underset{3}{\nu _2}{}^2 \underset{3}{\nu _3} \underset{2}{\nu _6}\, \varepsilon _{\{2,5,7\}}+60 \underset{3}{\nu _1}{}^2 \underset{3}{\nu _2}{}^2  \underset{3}{\nu _3} \underset{2}{\nu _4} \underset{2}{\nu _6}\,\varepsilon _{\{2,3,5,7\}}\\
    &\hskip10pt+60 \underset{3}{\nu _1}{}^2 \underset{3}{\nu _2}{}^2 \underset{3}{\nu _3}{}^2 \underset{2}{\nu _5} \underset{2}{\nu _6}\, \varepsilon _{\{2,4,5,7\}}+30 \underset{3}{\nu _1}{}^2 \underset{3}{\nu _2}{}^2 \underset{2}{\nu _4} \underset{2}{\nu _5} \underset{2}{\nu _6} \,\varepsilon _{\{2,6,7\}}\\
    &\hskip10pt+120 \underset{3}{\nu _1}{}^2 \underset{3}{\nu _2}{}^2 \underset{3}{\nu _3}{}^2 \underset{2}{\nu _4} \underset{2}{\nu _5} \underset{2}{\nu _6}\, \varepsilon _{\{2,3,4,5,7\}}.
    \end{align*}
By Proposition \ref{weak independent set enumeration}, the only weak independent set of size five in $H$ is  \hfill\break $\{v_2, v_3, v_4, v_5, v_7\}$.     
\end{example}

We can recover $k$-independent sets in a hypergraph $H$ containing no isolated vertices in the same way as weak independent sets in $H$ with the adjustment that all hyperedge labels are distinct zeons which have degree of nilpotency equal to $k+1$.
While we might not be able to find copies of a complete hypergraph using zeons, we are still able to find collections of pairwise adjacent vertices.  To do this, we construct an ordinary graph $G_I$ as follows:  Let the vertex set of $G_I$ be the same as that of $H$ with two vertices in $G_I$ being adjacent if and only if they are not adjacent in $H$.  Independent sets in this graph correspond to sets of pairwise adjacent vertices in $H$ and can be determined using the zeon clique representation of $G_I$.

\subsection{Matchings}\label{matchings subsection}
Recall that given a graph $G$, a $matching$ in $G$ is a subset of edges such that no pair of edges share a common vertex. A {\em k-matching} is a matching containing $k$ edges.  An $n$-matching on a graph containing $2n$ vertices is called a {\em perfect matching}. In other words, a perfect matching is a matching which contains all vertices of $G$.  

Here we generalize the zeon approach for graphs (see \cite{Staples2019}) to enumerate matchings in hypergraphs.

\begin{definition}  Given a hypergraph $H$, a {\em matching} in $H$ is a subset of hyperedges whose pairwise intersections are empty.  A {\em k-matching} is a matching containing $k$ hyperedges.  A {\em perfect matching} is a matching containing all vertices of $H$.  In particular, if $H$ is an $r$-uniform hypergraph on $rn$ vertices, then a perfect matching is a collection of $n$ disjoint hyperedges. 
\end{definition}

We begin by constructing a ``zeon incidence representation'' of $H$.  In the following, $\Gamma_H$  is still a sum of terms representing hyperedges of $H$ by using each hyperedge's incident vertices to determine the index set $J$.

\begin{definition}\label{zeon incidence rep}
If $H = (V,E)$ is a hypergraph on $n$ vertices, label the vertices with $\{1,2,\dots,n\}$ and assign the label $\zeta_J$ to each hyperedge $J \in E$.  Then the {\em zeon incidence representation} of $H$ is the zeon element  \begin{equation*}
\Gamma_H = \sum_{J \in E}\zeta_J.
\end{equation*}
\end{definition}

\begin{example}
The zeon incidence representation of the hypergraph $H$ in Figure \ref{first} is 
$\Gamma_H = \zeta_{\{1,2,3\}}+\zeta_{\{1,4,7\}}+\zeta_{\{1,4,5\}}+\zeta_{\{3,6\}}+\zeta_{\{4,6\}}+\zeta_{\{5,6\}}.$
\end{example}

\begin{proposition}\label{hypergraph matchings}
Let $H$ be a hypergraph on $n$ vertices with $\Gamma_H$ defined as above.  For $k \in \mathbb{N}$ we have \begin{equation*}
{\Gamma_H}^{k} = k!\sum_{I \in 2^{[n]}}\alpha_I\zeta_I
\end{equation*} where $\alpha_I$ is the number of $k$-matchings in $H$ on vertex set $I \subseteq V$.
\end{proposition}

\begin{proof}
By construction, $\Gamma_H$ is a sum, each term of which represents a hyperedge in $H$.  It follows that taking the $k$-th power of $\Gamma_H$ gives a sum whose terms each represent a $k$-subset of hyperedges of $H$.  The nilpotent property of the vertex labels guarantees that the only surviving terms in $\Gamma_{H}^{k}$ represent products of disjoint hyperedges, i.e. $k$-matchings in $H$.  The factor of $k!$ comes from the number of ways $\zeta_I$ appears in the product.
\end{proof}

\begin{example}
For the hypergraph in Figure \ref{first} we have \begin{equation*}
{\Gamma_H}^2 = 2 \left(\zeta_{\{1,2,3,4,6\}}+ \zeta_{\{1,2,3,5,6\}}+\zeta_{\{1,3,4,5,6\}}+\zeta_{\{1,3,4,6,7\}}+\zeta_{\{1,4,5,6,7\}} \right)
\end{equation*}
which shows there is one 2-matching on each vertex set appearing in the sum.  It is also apparent that ${\Gamma_H}^3=0$, so $H$ does not contain any 3-matchings.
\end{example}

\begin{corollary}
Let $H$ be an $r$-uniform hypergraph on $rn$ vertices and $\Gamma_H$ defined as above.  Let $\mu_H$ denote the number of perfect matchings in $H$ and let $\alpha_{[rn]}$ denote the coefficient of $\zeta_{[rn]}$ in ${\Gamma_H}^n$.  Then $$\mu_H = \frac{1}{n!} \alpha_{[rn]}.$$
\end{corollary}

\subsubsection{$j$-intersecting matchings}

We can further generalize the idea of matchings in hypergraphs.  To that end, a {\em $j$-intersecting matching} in a hypergraph $H=(V,E)$ is defined to be a subset of $E$ such that the pairwise intersections of elements in the subset contain at most $j$ elements\footnote{Note that Proposition \ref{hypergraph matchings} treats the case of $0$-intersecting matchings.}.  

To find $j$-intersecting matchings in a hypergraph, we will construct an ordinary graph $G_j=(E,E_j)$ whose vertex set is the hyperedge set of $H$ . Two vertices in $G_j$ are made adjacent if and only if the corresponding hyperedges of $H$ intersect in {\bf at least} $j+1$ vertices. Now independent sets in $G_j$ correspond to $j$-intersecting matchings in $H$ and can be determined as described in the following result.

\begin{proposition}
Let $H=(V,E)$ be a hypergraph and let $G_j$ be the graph whose vertex set is the hyperedge set of $H$ with two vertices being adjacent if and only if the corresponding hyperedges of $H$ intersect in at least $j+1$ vertices.  Let $\Phi_{G_j}$ be the independent set representation of $G_j$.  Then for $k \in \mathbb{N}$ we have \begin{equation*}
{\Phi_{G_j}}^k=\sum_{I \subset E, |I|=k}\zeta_{N(I)}\varepsilon_I,
\end{equation*} where \begin{equation*}
\varepsilon_I = \prod_{i \in I}\varepsilon_i
\end{equation*} is a basis blade of $\I_{|E|}$ whose index set represents a $j$-intersecting matching on $k$ hyperedges in $H$.
\end{proposition}

\begin{proof}
By construction of $G_j$, independent sets in $G_j$ are subsets of hyperedges of $H$ such that the pairwise intersection of any hyperedges in the subset contains at most $j$ vertices, i.e. a $j$-intersecting matchings in $H$.  By Theorem \ref{ind sets}, $\Phi_{G_j}{}^k$ reveals the independent sets of $G_j$ of size $k$ and hence the $j$-intersecting matchings on $k$ hyperedges in $H$.
\end{proof}

\subsection{Minimal Vertex Coverings and Transversals}
\begin{definition}
Let $H = (V,E)$ be a hypergraph.  A {\em vertex cover} or {\em transversal} of $H$ is a subset $T \subseteq V$ such that for all hyperedges $e \in E$ we have $T \cap e \neq \emptyset$.  A transversal of $H$ is a minimum transversal if it is not properly contained in any other transversal of $H$.
\end{definition}

As mentioned previously, an important problem in hypergraph theory is the generation of all minimum transversals of a hypergraph $H$.  To generate minimum transversals of a hypergraph, we will construct an object which is similar to the independent set representation.  In this case, however, we will use idem-Clifford generators as hyperedge labels since we do not wish to remove any particular configurations, but instead wish to ensure that each hyperedge is covered.

Letting $\{\varepsilon_{1},\ldots, \varepsilon_{m}\}$ be generators of $\I_m$ and $\{x_1, \ldots, x_n\}$ be generators of $I_n$, it will be useful to consider the idem-Clifford algebra $\I_m\otimes\I_n$.  Note that elements of  $\I_m\otimes\I_n$ satisfy the following for $I,K\in 2^{[m]}$ and $J,L\in2^{[n]}$:
\begin{align*}
(\varepsilon_I\,x_J)(\varepsilon_K\,x_L)&=(\varepsilon_K\,x_L)(\varepsilon_I\,x_J)\\
&=\varepsilon_{I\cup K}\,x_{J\cup L}.
\end{align*}
In particular, $(\varepsilon_I\,x_J)^2=\varepsilon_I\,x_J$, and any element $u\in \I_m\otimes\I_n$ has a canonical expansion of the form \begin{equation*}
u=\sum_{I\in 2^{[m]}, J\in 2^{[n]}} u_{(I,J)}\,\varepsilon_I\,x_J,
\end{equation*}
where $u_{(I,J)}\in\mathbb{R}$ for each multi-index pair $(I,J)$.

Let $H=(V,E)$ be a hypergraph with $n$ vertices and $m$ hyperedges.  Note that any isolated vertices of $H$ cannot appear in a transversal, so we will remove and isolated vertices from the hypergraph before performing the following construction.  Letting $\{\varepsilon_{1},\ldots, \varepsilon_{m}\}$ be generators of $\I_m$ and $\{x_1, \ldots, x_n\}$ be generators of $\I_n$, consider the idem-Clifford algebra $\I_m\otimes\I_n$.  Let $\varphi: E \to \{\varepsilon_1,\varepsilon_2,\dots,\varepsilon_m\}$ be a labeling of the hyperedges and let $$\psi(v) = \prod_{\text{$e$ incident to $v$}}\varphi(e).$$  Let $\lambda: V \to \{x_1,x_2,\dots,x_n\}$  be a labeling of the vertices and define 
\begin{equation}\label{idem transversal rep}
\sigma = \sum_{v \in V}\psi(v)\lambda(v).
\end{equation}

\begin{definition}
Let $H=(V,E)$ be a hypergraph with $n$ vertices and $m$ hyperedges. The {\em idem-Clifford transversal representation of $H$} is the element $\sigma\in\I_m\otimes\I_n$ defined by \eqref{idem transversal rep}.
\end{definition}

\begin{proposition}\label{min transversals prop}
Let $H=(V,E)$ be a hypergraph with $m$ hyperedges and let $\sigma$ be the idem-Clifford transversal representation of $H$.  Then there is some smallest integer $k$ such that $\sigma^k$ contains at least one nonzero term of the form $\alpha_I\varepsilon_{[m]}x_I$ where $I$ represents the vertex set of a minimum transversal of $H$ and $|I| = k$. Each term of this form appearing in $\sigma^k$ represents a distinct minimum transversal of $H$.
\end{proposition}

\begin{proof}
Since hyperedges are nonempty, we see that the coefficients for $\sigma$ represent all hyperedges of $H$.  Note that when taking powers of $\sigma$, we obtain a polynomial whose idem-Clifford coefficients have index set equal to the union of the index sets of the idem-Clifford coefficients being multiplied.  It is evident that there is some power of $\sigma$ such that $\varepsilon_{[m]}$ appears as an idem-Clifford coefficient on some term $x_I$.  Let $k$ be the smallest such power.  By construction, terms in $\sigma^k$ having $\varepsilon_{[m]}$ as an idem-Clifford coefficient contain $x_I$ where $I$ is a set of $k$ vertices which intersect each hyperedge of $H$.
\end{proof}

\begin{example}
The idem-Clifford transversal representation for the hypergraph $H$ of Figure \ref{first} is
\begin{align*}
\sigma &= \varepsilon_{\{1,2,3\}}\,x_1+\varepsilon_{1}\,x_2 + \varepsilon_{\{1,4\}}\,x_3 + \varepsilon_{\{2,3,5\}}\,x_4\\
&\hspace{10pt} + \varepsilon_{\{3,6\}}\,x_5 + \varepsilon_{\{4,5,6\}}\,x_6 + \varepsilon_{2}\,x_7.
\end{align*}

A quick {\em Mathematica} calculation reveals \begin{align*}
\sigma^2&=2\,\varepsilon _{\{1,2\}}\,x_{\{2,7\}}+ \varepsilon _{\{1,4\}}\,x_{3}+2\,\varepsilon _{\{1,4\}}\,x_{\{2,3\}}+ \varepsilon _{\{3,6\}}\,x_{5}+\varepsilon _{\{1,2,3\}}\,x_{1}\\
&\hskip10pt+2\,\varepsilon _{\{1,2,3\}}\,x_{\{1,2\}}+2\,\varepsilon _{\{1,2,3\}}\,x_{\{1,7\}}+2\,\varepsilon _{\{1,2,4\}}\,x_{\{3,7\}}+2 \,\varepsilon _{\{1,3,6\}}\,x_{\{2,5\}}\\
&\hskip10pt+\varepsilon _{\{2,3,5\}}\,x_{4}+2 \,\varepsilon _{\{2,3,5\}}\, x_{\{4,7\}}+2 \,\varepsilon _{\{2,3,6\}}\,x_{\{5,7\}}+ \varepsilon _{\{4,5,6\}}\,x_{6}\\
&\hskip10pt+2\, \varepsilon _{\{1,2,3,4\}}\,x_{\{1,3\}}+2\, \varepsilon _{\{1,2,3,5\}}\,x_{\{1,4\}}+2\, \varepsilon _{\{1,2,3,5\}}\,x_{\{2,4\}}+2\,\varepsilon _{\{1,2,3,6\}}\,x_{\{1,5\}}\\
&\hskip10pt+2\, \varepsilon _{\{1,3,4,6\}}\,x_{\{3,5\}}+2\, \varepsilon _{\{1,4,5,6\}}\,x_{\{2,6\}}+2\, \varepsilon _{\{1,4,5,6\}}\,x_{\{3,6\}}+2\,\varepsilon _{\{2,3,5,6\}}\,x_{\{4,5\}}\\
&\hskip10pt+2\,\varepsilon _{\{2,4,5,6\}}\,x_{\{6,7\}}+2\, \varepsilon _{\{3,4,5,6\}}\,x_{\{5,6\}}+2\, \varepsilon _{\{1,2,3,4,5\}}\,x_{\{3,4\}}\\
&\hskip10pt+2\, \varepsilon _{\{2,3,4,5,6\}}\,x_{\{4,6\}}+2\,\varepsilon _{\{1,2,3,4,5,6\}}\,x_{\{1,6\}}+\varepsilon _{1}\, x_{2}+\varepsilon _{2}\,x_{7}.
\end{align*}
The presence of term $2\,\varepsilon_{\{1,2,3,4,5,6\}}\,x_{\{1,6\}}$ indicates that $\{v_1,v_6\}$ is a minimum transversal of $H$.  Further, this is the only term in $\sigma^2$ having $\varepsilon_{\{1,2,3,4,5,6\}}$ as a coefficient, so $\{v_1,v_6\}$ is the only minimum transversal of $H$.
\end{example}

In light of Proposition \ref{min transversals prop}, it is evident that the problem of finding all minimum transversals of a hypergraph $H$ is equivalent to finding the smallest $k$ such that $\sigma^k$ contains a term with an idem-Clifford coefficient whose index set is of size $m$ on at least one $x_I$.

\section{Zeon Formulations of Other Hypergraph Problems}\label{zeon formulations}

Some open problems related to hypergraphs now lead to open problems involving zeons.  The following hypergraph conjectures are attributed to Ryser~\cite{Ryser} and Frankl~\cite{Frankl}.  The interested reader can find these and other open problems at the Open Problem Garden website~\cite{openproblemgarden}.

\subsection*{Ryser's conjecture}

The following conjecture first appeared in 1971 in the Ph.D. thesis of J. R. Henderson, whose advisor was Herbert John Ryser~\cite{Henderson1971}.

\begin{conjecture*}[Ryser]
Let $H$ be an $r$-uniform $r$-partite hypergraph. If $\varkappa$ is the maximum number of pairwise disjoint hyperedges in $H$, and $\tau$ is the size of the smallest set of vertices which meets every hyperedge, then $\tau \le (r-1) \varkappa$.
\end{conjecture*}

In terms of the zeon incidence representation (Definition \ref{zeon incidence rep}), one sees that $\varkappa=\kappa(\Gamma)-1$.  The quantity $\tau$ is  then revealed as the minimal grade among basis blades $\zeta_I$ that ``annihilate'' $\Gamma$, i.e., blades that satisfy $\zeta_I\Gamma=0$.  Hence, the following formulation is offered.

\begin{conjecture}[Zeon formulation of Ryser's conjecture I]
Let $\Gamma$ be the zeon incidence representation of an $r$-uniform $r$-partite hypergraph, and let $\gamma$ be the zeon element defined by \begin{equation*}
\gamma=\sum_{\{I: \zeta_I\Gamma=0\}}\zeta_I.
\end{equation*}
Then,\begin{equation*} 
\natural\gamma \le (r-1) (\kappa(\Gamma)-1),
\end{equation*}
where $\natural\gamma$ denotes the minimum grade of $\gamma$ and $\kappa(\Gamma)$ is the index of nilpotency of $\Gamma$.
\end{conjecture}

\subsection*{Frankl's union-closed sets conjecture}

P\'{e}ter Frankl stated the next conjecture in terms of intersection-closed set families in 1979.   

\begin{conjecture*}[Frankl]
Let $\mathcal{F}$ be a finite family of finite sets, not all empty, that is closed under taking unions.  Then there exists $x$ such that $x$ is an element of at least half the members of $\mathcal{F}$.
\end{conjecture*}

In terms of hypergraphs, Frankl's conjecture is equivalent to the following.  Let $H$ be a hypergraph on $n$ vertices $V=\{v_1, \ldots, v_n\}$ and $m$ hyperedges $E=\{e_1, \ldots, e_m\}$ having the property that $e_i, e_j\in E$ implies $e_i\cup e_j\in E$.  For convenience, let such a hypergraph be said to satisfy {\em condition $\mathfrak{F}$}.  According to the conjecture, when $H$ satisfies condition $\mathfrak{F}$, there exists a vertex $v_i\in V$ such that $v_i$ is incident with at least $m/2$ hyperedges.  Equivalently, multiplication by $\zeta_{\{i\}}$ annihilates at least $m/2$ hyperedges from $\Gamma_H$.

\begin{conjecture}[Zeon formulation of Frankl's conjecture]
Let $\Gamma_H$ be the zeon incidence representation of a hypergraph $H$ satisfying condition $\mathfrak{F}$.  Then, there exists $i\in[n]$ such that \begin{equation*}
\langle\langle\zeta_{i}\,\Gamma_H\rangle\rangle\le \frac{1}{2}\langle\langle\Gamma_H\rangle\rangle.
\end{equation*}
\end{conjecture}

\section{Conclusion and Avenues for Future Work}\label{conclusion}
As we have shown, many zeon and idem-Clifford algebraic results used to enumerate structures in graphs can be generalized to the hypergraph setting. Further, given that zeon-algebraic methods have been used in graph coloring problems \cite{StaplesStellhorn2017}, another promising avenue for future research is applying zeon and idem-Clifford algebraic techniques to hypergraph colorings.  Moreover, inverses of zeon matrices have recently been shown to enumerate paths in graphs~\cite{zeon laplacian}, opening another potential generalization to hypergraphs.

\end{document}